\documentclass{amsart}
\usepackage[marginparwidth=60pt]{geometry} 
\usepackage{graphicx}
\usepackage{lipsum}

\calclayout 

\usepackage[utf8]{inputenc}
\usepackage[british]{babel}
\usepackage{microtype}
\usepackage[pdfencoding=auto,pdfusetitle]{hyperref}
\hypersetup{colorlinks=true, linkcolor=black, citecolor=black, urlcolor=black}

\usepackage{amsmath,amssymb,amsfonts,amsbsy}
\usepackage[mathcal]{euscript}
\usepackage{mathrsfs,dsfont,bbm}
\usepackage{fdsymbol,marvosym}
\usepackage{relsize}

\usepackage{amsthm}
\usepackage{mathtools}
\usepackage[noabbrev,capitalise]{cleveref}

\usepackage{graphicx}
\usepackage[all]{xy}

\usepackage[svgnames]{xcolor}
\definecolor{blue(munsell)}{rgb}{0.0, 0.5, 0.69}
\usepackage{svg}

\usepackage{verbatim,adjustbox,float,subfig,framed,caption,epigraph}
\usepackage{stackrel}
\usepackage{enumitem}

\usepackage{tikz}
\usepackage{quiver}

\usepackage{circledsteps}

\numberwithin{equation}{section}

\theoremstyle{plain}
\newtheorem{theorem}{Theorem}[section]
\newtheorem{lemma}[theorem]{Lemma}
\newtheorem{proposition}[theorem]{Proposition}
\newtheorem{corollary}[theorem]{Corollary}

\theoremstyle{definition}
\newtheorem{definition}[theorem]{Definition}

\newtheorem{remark}[theorem]{Remark}

\newtheorem{example}[theorem]{Example}
\newtheorem{examples}[theorem]{Examples}

\newcommand{\K}{\mathcal{K}}
\newcommand{\catk}{\mathcal{K}}
\newcommand{\catl}{\mathcal{L}}
\newcommand{\catd}{\mathcal{D}}
\newcommand{\cata}{\mathcal{A}}

\newcommand{\Set}{\mathcal{S}et}
\newcommand{\Rel}{\mathcal{R}el}
\newcommand{\Ab}{\mathcal{A}b}

\newcommand{\Pos}{\mathcal{P}os}

\newcommand{\Mod}{\mathcal{M}od}
\renewcommand{\Vec}{\mathcal{V}ec}

\newcommand{\Bool}{\mathcal{B}ool}
\newcommand{\Met}{\mathcal{M}et}
\newcommand{\CMet}{\mathcal{CM}et}

\newcommand{\id}{\mathrm{id}}
\def\op{\mathrm{op}}

\newcommand{\gr}{\text{grade\hspace*{0.15mm}}}
\newcommand{\rk}{\text{grade\hspace*{0.15mm}}}

\renewcommand{\epsilon}{\varepsilon}
\renewcommand{\phi}{\varphi}

\usepackage[all]{xy}

\title{A finitary adjoint functor theorem}

\author{Ji\v r\'{\i} Ad\'amek}
\address{
Department of Mathematics, Faculty of Electrical Engineering\\
Czech Technical University in Prague, Czech Republic\newline
\indent Institute for Theoretical Computer Science, Technical University of Braunschweig, Germany
}
\email{j.adamek@tu-bs.de}
\thanks{Supported by the Grant Agency of the Czech Republic: Grant 22-02964S}

\author{Lurdes Sousa}
\address{University of Coimbra, Department of Mathematics, CMUC, 3000-143 Coimbra, Portugal  \newline \indent Polytechnic  of Viseu, ESTGV, Portugal}
\email{sousa@estv.ipv.pt}
\thanks{Partially supported by the Centre for Mathematics of the University of Coimbra (funded by the Portuguese Government through FCT/MCTES, DOI 10.54499/UIDB/00324/2020)}



\date{\today}

\begin{document}

\maketitle

{\hfill \dedicatory{\em To the memory of V\v era Trnkov\'a}}

\begin{abstract}Graduated locally finitely presentable  categories are introduced, examples include categories of sets, vector spaces, posets, presheaves and Boolean algebras. A finitary functor between graduated locally finitely presentable  categories is proved to be a right adjoint if and only if it preserves countable limits. For endofunctors on vector spaces or pointed sets even countable products are sufficient. Surprisingly, for set functors there is a single exception of a (trivial) finitary functor preserving countable products  but not countable limits.
\end{abstract}

\section{Introduction}
A functor between locally presentable categories is a right adjoint iff it is accessible and preserves limits \cite[Thm. 1.66]{AR}. We introduce  a wide class of locally finitely presentable categories, called {\em graduated}, and prove that a finitary functor between them is a right adjoint iff it preserves countable limits. Graduation essentially means that every finitely presentable object is assigned a grade in $\mathds{N}$ so that  proper subobjects and  proper strong quotients have lower grades. Examples of graduated categories include categories of
\begin{enumerate}
\item[(1)] sets, posets, Boolean algebras, $M$-sets for finite monoids $M$, and left modules over finite semirings;
\item[(2)] vector spaces, presheaves in $\Set^{\cata^{\op}}$ where $\cata$ has finite connected components, and relational structures of finitary signatures.
\end{enumerate}
Our paper has been inspired by Tendas  who proved the following result for locally finitely presentable categories having (a) only countable many finitely presentable objects (up to isomorphism) and (b) finite hom-sets for them: a finitary functor between such categories preserves limits iff it preserves countable limits \cite[Remark 2.10]{GT}. The examples in (1) above satisfy these conditions, those of (2) do not in general.
Besides, our proof (completely different from that of Tendas) can also be used to include the categories of metric spaces and complete metric spaces to our list of examples.

A second inspiration of our paper is Trnkov\'a's result concerning the question when functors preserving products automatically preserve limits \cite{T3}.
Can one reduce countable limits to countable products? The answer is affirmative for endofunctors of categories such as vector spaces or pointed sets. Surprisingly such a  reduction is almost, but not completely, possible for set functors. Indeed, the functor
$$C_{01}\text{ defined by } \emptyset \mapsto \emptyset \; \; \text{and}\; \; X \mapsto1 \; \; \text{for all $X\not= \emptyset$}$$
preserves all products but not countable limits. This is the single exception: a finitary set functor preserving countable products but not countable limits is naturally isomorphic to $C_{01}$.

\vskip2mm
\noindent \textbf{Acknowledgement.} The authors are grateful to Giacomo Tendas for useful discussions.

\section{Graduated categories}

In this section graduated locally finitely presentable categories are introduced and examples are presented.  In the subsequent section we prove that a finitary functor between graduated categories is a right adjoint iff it preserves countable limits.

\begin{remark}\label{rem:graduated}   The following properties of  locally finitely presentable categories $\K$ are used in the proof of our main theorem:

\begin{enumerate}
\item[(1)] $\K$ is complete and cocomplete (\cite[Rem. 1.56]{AR}).

\item[(2)] $\K$ has (strong epi, mono)-factorizations (\cite[Prop. 1.62]{AR}).

	\item[(3)] There is only a set of finitely presentable objects up to isomorphism.
	
	\item[(4)] For every directed colimit $c_i :C_i \to C\,  (i\in I)$ in $\K$ and every finitely presentable object $K$ each morphism from $K$ to $C$ factorizes through some $c_i$.
	
	\item[(5)] Every object of $\K$ is a directed colimit of finitely presentable objects.
\end{enumerate}
	
	Moreover, we are going to require the following property (that most of ``everyday" locally finitely presentable categories have, but not all):
	
	\begin{enumerate}
\item[(6)] Every subobject and every strong quotient of a finitely presentable object is finitely presentable.
\end{enumerate}
	
\end{remark}

\begin{definition} A locally  finitely presentable   category  is {\em graduated} if   to every  finitely presentable  object $A$ a natural number
$$\gr A$$
({\em the grade}) is  assigned satisfying the following:

\vspace{1mm}

\hspace{15mm}\begin{tabular}{p{11cm}}Every (proper) subobject and every (proper) strong quotient of $A$ is finitely presentable and has grade at most (smaller than, resp.) $\rk A$.\end{tabular}

\end{definition}

\begin{remark}\label{rem:rank}
In particular, isomorphic finitely presentable  objects have the same grade. Moreover, if
 $A$ and  $B$ are finitely presentable objects of the same grade, every monomorphism and every strong epimorphism between them is invertible.
 \end{remark}

\begin{examples}
The following categories are graduated.
\begin{enumerate}
\item $\Set$ and $\Set_p$, the category of pointed sets. Put
$$\gr A=\text{card} A.$$
\item The presheaf category $\Set^{\cata^{\op}}$ where $\cata$ has finite connected components. A presheaf $A\colon \cata^{\op} \to \Set$ is finitely presentable iff the sets $As$ ($s\in \text{obj } \cata$) are finite, and all but finitely many are empty.  (Indeed, the above condition implies that $A$ is finitely presentable due to the object-wise computation of directed colimits of presheaves. Conversely, given a finitely presentable presheaf $A$, let $A_i\, (i\in I)$ be the collection of all subfunctors  mapping all but finitely many components of $\cata^{\op}$ to the empty set. Each $A_i$ fulfils the above condition. Since $A$ is a directed colimit of that collection, it is one of those subfunctors.)

Put
$$\gr  A=\sum_{s\in \cata}\text{card} As.$$
\item $\Pos$, the category of posets. For the graduation we apply the lexicographic order on $\mathbb{N}^2$ and use the induced subposet $\hat{\mathbb{N}}$ of all pairs $(n,k)\in \mathbb{N}^2$ with $k\leq n^2$. This poset is isomorphic to $\mathbb{N}$ under the mapping $\phi\colon \hat{\mathbb{N}}\to \mathbb{N}$ assigning to $(n,k)$ the number of all smaller members of $\hat{\mathbb{N}}$:
$$\begin{array}{lccccc}(n,k)&(0,0)&(1,0)&(1,1)&(2,0)&\dots\\ \hline
\phi(n,k)&0&1&2&3&\dots\end{array}$$
The grade of a poset $(X,R)$, where $R\subseteq X^2$ is the order relation, is 
$$\gr  (X,R)=\phi(|X|,|R|).$$

Given a proper subobject $(X',R')\rightarrowtail (X,R)$ we either have $|X'|<|X|$, or $|X'|=|X|$ and $|R'|<|R|$; thus $\gr (X',R')<\gr (X,R)$.

Consider a strong quotient: it is easy to see that it is invertible in $\Pos$ iff it is carried by a bijection. Thus given a proper strong quotient $(X,R) \twoheadrightarrow (X',R')$, we have $|X'| < |X|$, which yields $\gr (X',R') < \gr (X,R)$.

\item $\Bool$, the category of Boolean algebras. Every finitely presentable Boolean algebra is finite and we put
$$\rk A=\text{card}\, A.$$
\item
$\Omega$-$\Rel$, the category of relational structures of a signature $\Omega=(\Omega_n)_{n\in \mathds{N}}$. Objects are pairs $A=(X, (\omega_A))$ consisting of a set $X$  with relations $\omega_A\subseteq X^n$ for all $\omega \in \Omega_n$. Finitely presentable objects are such that both $X$ and $\coprod_{\omega\in \Omega}\omega_A$ are finite sets. Put
$$\rk A=\text{card}\, X+\sum_{\omega \in \Omega}\text{card}\, \omega_A.$$
 If $B$ is a proper subobject of $A$, and has the same elements, then $\omega_B\not\subseteq \omega_A$ for some $\omega$, thus $\rk B<\rk A$. This inequality also holds if $B$ has less elements than $A$. The argument for proper quotients $B$ is similar: a strong quotient $e\colon A \twoheadrightarrow B$ whose underlying map is bijective is indeed an isomorphism in $\Omega$-$\Rel$.

\item $M$-$\Set$, the category of sets with an action of $M$, for all finite monoids $M$. Every finitely presentable $M$-set $A$ is finite, and we put
$$\gr  A=\text{card} A.$$
In contrast, $M$-$\Set$ is not graduated for the monoid $M=(\mathds{N},+,0)$: That monoid defines a finitely presentable $M$-set $\mathds{N}$ (with monoid action given by addition). The proper $M$-subset $\mathds{N}-\{0\}$ is isomorphic to it, so it cannot have a lower grade.
\item $S$-$\Mod$, the category of left modules, for every finite semiring $S$: here also $\gr  A=\text{card} A.$ Since  free finitely generated semirings are finite, also all finitely presentable objects are finite.

Again this does not hold for infinite semirings. For example the category $\Ab=\mathds{Z}$-$\Mod$ of abelian groups is not graduated: the proper subobject $2\mathds{Z}\hookrightarrow \mathds{Z}$ fulfils $2\mathds{Z}\cong \mathds{Z}$.

\item $K$-$\Vec$, the category of vector spaces over a field $K$.  Put
$$\gr  A=\text{dim} A.$$
\end{enumerate}
\end{examples}

\section{The finitary adjoint functor theorem}

For every locally  finitely presentable   category  $\K$ we denote by $\K_{fp}$ the full subcategory of all finitely presentable objects.

\begin{lemma}\label{lem:fp_subobjs} Every object $K$ of a graduated locally  finitely presentable   category  is the directed colimit of the diagram of all its  finitely presentable  subobjects.
\end{lemma}

\begin{proof} Since our  category  $\catk$ is locally finitely presentable, $K$ is the canonical filtered colimit of the diagram
$$D_K\colon \catk_{fp}\downarrow K\to \catk , \; \; (A\xrightarrow{a}K)\mapsto A$$
(see \cite[Prop. 1.22]{AR}). Let $m_a\cdot a'=a$ be a (strong epi, mono)-factorization for each $a\colon A\to K$:
\[\begin{tikzcd}
	A &&& {} & B \\
	& {A'} && {B'} \\
	&& K
	\arrow["{f^{\prime}}", dashed, tail, from=2-2, to=2-4]
	\arrow["{m_a}"', tail, from=2-2, to=3-3]
	\arrow["{m_b}", tail, from=2-4, to=3-3]
	\arrow["{a'}"', two heads, from=1-1, to=2-2]
	\arrow["f", from=1-1, to=1-5]
	\arrow["{b'}", two heads, from=1-5, to=2-4]
\end{tikzcd}\]
For every connecting morphism $f\colon (A,a)\to (B,b)$ of $D_K$ the diagonal fill-in property yields a corresponding monomorphism $f'\colon A'\to B'$.

We thus obtain a diagram $D'_K$ of objects $A'$ and connecting morphisms $f'$. For each  finitely presentable object $A$ the strong  epimorphism $a'$  proves  that $A'$ is  finitely presentable  (since $\catk$ is graduated). Thus $D'_K$ is a directed diagram of  finitely presentable  subobjects of $K$.

Conversely, every  finitely presentable  subobject $m'\colon A'\to A$ has the form $m_a$ for $a=m'$. Thus $D'_K$ is the diagram of \emph{all} finitely presentable  subobjects of $K$.  Its colimit is, obviously, $m_a\colon A'\to K$ for $(A,a)\in \catk_{fp}\downarrow \catk$.
\end{proof}

\begin{remark}\label{rem:count_codir} Let $I$ be a countably   codirected  poset: every countable subset has a lower bound.

(1) Given a decomposition $I=\bigcup_{k\in \mathds{N}}I_k$, some $I_k$ is \emph{initial}, i.e. every element of $I$ lies over some element of $I_k$. Indeed, assuming the contrary, for each $k$ we have a counter-example $i_k \in I$ not majorizing elements of $I_k$. The countable set $\{i_k\}_{k\in \mathds{N}}$ has a lower bound $j\in I$. But this is a contradiction: $j\not\in I_k$ for any $k$.

(2) Given a diagram $D\colon I\to \catk$, for every initial subset $J\subseteq I$ the   limits of $D$ and of its restriction $D/J\colon J\to \catk$, are the same. More precisely: the limit cone $\pi_i\colon L\to Di \; (i\in I)$ of $D$ yields a limit cone $\pi_j\colon L_j\to Dj\; (j\in J)$ of $D/J$, and vice versa.
\end{remark}

\begin{theorem}\label{thm:main} Let $F\colon \catk\to  \catl$ be a  finitary  functor between locally  finitely presentable  categories with $\catk$ graduated. Then $F$ is a right adjoint if and only if it preserves countable limits.
\end{theorem}

\begin{proof} (1) By the Adjoint Functor Theorem \cite[Thm. 1.66]{AR}, it is sufficient to prove that $F$ preserves limits. We prove below that it preserves countably   codirected  limits. This is sufficient: it is easy to see that the limit of every diagram $D\colon  \catd\to \catk$ is a countably codirected  limit of limits of diagrams $D/ \catd^{\prime}$, where $\catd^{\prime}$ ranges over countable subcategories of $\catd$.

(2) Let $I$ be a countably   codirected  poset and $D=(D_i)_{i\in I}$ a diagram in $\catk$ with a limit cone $(\pi_i)_{i\in I}$:
\begin{equation*}
	\begin{tikzcd}
	& L && {} \\
	{D_i} && {D_j} & {}
	\arrow["{d_{ij}}"', from=2-1, to=2-3]
	\arrow["{\pi_i}"', from=1-2, to=2-1]
	\arrow[""{name=0, anchor=center, inner sep=0}, "{\pi_j}", from=1-2, to=2-3]
	\arrow[""{name=1, anchor=center, inner sep=0}, "{(i\leq j \text{ in } I)}", draw=none, from=1-4, to=2-4]
	\arrow["\dots"{description, pos=0.7}, draw=none, from=0, to=1]
\end{tikzcd}\end{equation*}
For every cone in $\catl$
\[\begin{tikzcd}
	& Q && {} \\
	{FD_i} && {FD_j} & {}
	\arrow["{Fd_{ij}}"', from=2-1, to=2-3]
	\arrow["{q_i}"', from=1-2, to=2-1]
	\arrow[""{name=0, anchor=center, inner sep=0}, "{q_j}", from=1-2, to=2-3]
	\arrow[""{name=1, anchor=center, inner sep=0}, draw=none, from=1-4, to=2-4]
	\arrow["\dots"{description}, draw=none, from=0, to=1]
\end{tikzcd}\]
we prove that a unique factorization through $(F\pi_i)$ exists.

 We can restrict ourselves to cones with $Q$  finitely presentable  in $\catl$. Indeed, since $\catl_{fp}$ is a dense subcategory of $\catl$, that result then extends to all cones of $FD$.

\vskip1mm

(2a) Existence. First we show that for every morphism $q\colon Q \to FK$ with $K\in \catk$ and $Q$ finitely presentable there is a least subobject $m\colon M \rightarrowtail K$  with $M$ finitely presentable such that $q$ factorizes through $Fm$. For that, we  express
 $K$ as a directed colimit of all its finitely presentable subobjects   (Lemma \ref{lem:fp_subobjs}) and use that $F$ preserves that colimit. Thus $q$ factorizes through $Fm\colon FM\to FK$ for some subobject $m\colon M\rightarrowtail K$ with $M\in\catk_{fp}$. We claim that there exists a least such subobject: one contained in every subobject $m'\colon M' \rightarrowtail K$ with $M'\in\catk_{fp}$ such that $q$ factorizes through $Fm'$.

Indeed, first choose an arbitrary  finitely presentable  subobject $m_0\colon M_0\rightarrowtail K$  such that  $q=Fm_0\cdot u_0$ for some $u_0\colon Q\to FM_0$. If $m_0$ is not the least one, then there exists a  finitely presentable  subobject $m\colon M\rightarrowtail K$  such that
$$q=Fm\cdot  u \text{ (for some $u$)  and } m_0\not\subseteq m.$$
Form the intersection $m_1$ of $m_0$ and $m$ as follows:
\[\begin{tikzcd}
	& {M_1} \\
	{M_0} && M \\
	& K
	\arrow["{m'_0}"', from=1-2, to=2-1]
	\arrow["{m'}", from=1-2, to=2-3]
	\arrow["{m_0}"', from=2-1, to=3-2]
	\arrow["m", from=2-3, to=3-2]
	\arrow["{m_1}", from=1-2, to=3-2]
\end{tikzcd}\]
Since $F$ preserves this pullback and $Fm_0\cdot u_0=Fm\cdot u$, we see that $q$ factorizes through $Fm_1$:
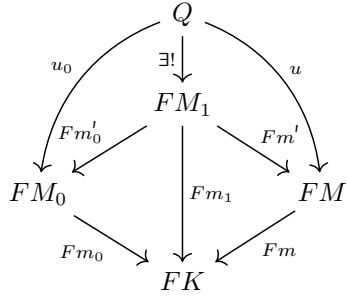
\[\begin{tikzcd}
	& Q \\
	& {FM_1} \\
	{FM_0} && FM \\
	& FK
	\arrow["{Fm'_0}"', from=2-2, to=3-1]
	\arrow["{Fm'}", from=2-2, to=3-3]
	\arrow["{Fm_0}"', from=3-1, to=4-2]
	\arrow["Fm", from=3-3, to=4-2]
	\arrow["{Fm_1}", from=2-2, to=4-2]
	\arrow["{\exists!}"', from=1-2, to=2-2]
	\arrow["{u_0}"', curve={height=18pt}, from=1-2, to=3-1]
	\arrow["u", curve={height=-18pt}, from=1-2, to=3-3]
\end{tikzcd}\]
Since $m_0\not\subseteq m_1$, we know that $m'_0$ is not invertible. Therefore, $M_1$ is a proper subobject of $M_0$ and we get
$$\gr  M_1< \gr  M_0.$$
We now iterate this procedure: either $M_1$ is the desired least subobject, or we find $M_2$ with $\gr  M_2<\gr   M_1$, etc. After less than $\gr   M_0$ steps we obtain the desired least subobject.

\vskip0.7mm

For each $i\in I$ let $m_i\colon M_i\to D_i$ be the least subobject with $M_i$ finitely presentable  such that
$$q_i=Fm_i\cdot r_i \; \text{ for some } r_i\colon Q\to FM_i.$$
Then the sets
$$I_n=\{i\in I; \, \gr  M_i=n\}$$
fulfil $I=\displaystyle{\bigcup_{n\in \mathds{N}}I_n}$. By Remark \ref{rem:count_codir}, some $I_k$ is initial in $I$. Thus the diagram $D_0=(D_i)_{i\in I_k}$ has the same limit as $D$.

Next we prove that each connecting morphism $d_{ij}\colon D_i\to D_j\; (i\leq j\, \text{ in $I_k$})$ of $D_0$ restricts to a morphism $m_{ij}\colon M_i\to M_j$. That is, we have a commutative square as follows:
\[\begin{tikzcd}
	{M_i} & {M_j} \\
	Di & Dj
	\arrow["{m_i}"', from=1-1, to=2-1]
	\arrow["{m_j}", from=1-2, to=2-2]
	\arrow["{m_{ij}}", from=1-1, to=1-2]
	\arrow["{d_{ij}}"', from=2-1, to=2-2]
\end{tikzcd}\]
Let us form a (strong epi, mono)-factorization of $d_{ij}\cdot m_i$ on the left:
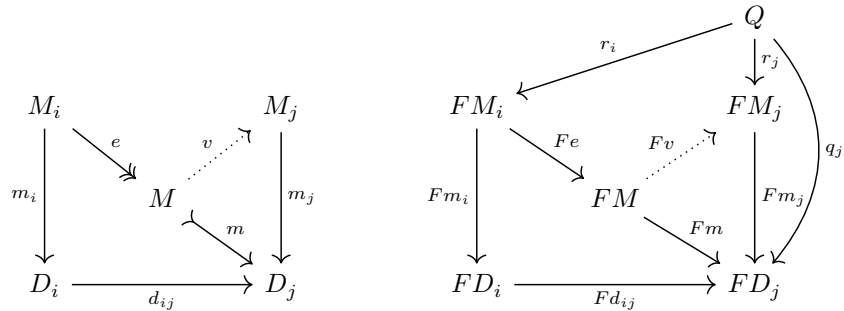
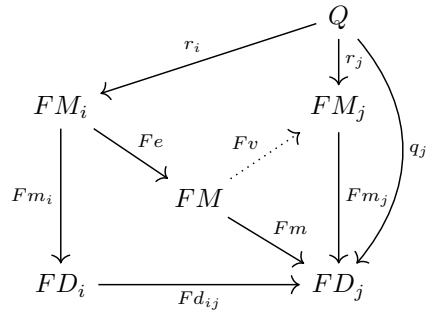
\[\begin{tikzcd}
	&&&&&& Q \\
	{M_i} && {M_j} && {FM_i} && {FM_j} \\
	& M &&&& FM \\
	{D_i} && {D_j} && {FD_i} && {FD_j}
	\arrow["e", two heads, from=2-1, to=3-2]
	\arrow["v", dotted, from=3-2, to=2-3]
	\arrow["{m_j}", from=2-3, to=4-3]
	\arrow["m", tail, from=3-2, to=4-3]
	\arrow["{m_i}"', from=2-1, to=4-1]
	\arrow["{d_{ij}}"', from=4-1, to=4-3]
	\arrow["{Fd_{ij}}"', from=4-5, to=4-7]
	\arrow["{Fm_i}"', from=2-5, to=4-5]
	\arrow["Fe", from=2-5, to=3-6]
	\arrow["Fv", dotted, from=3-6, to=2-7]
	\arrow["{Fm_j}", from=2-7, to=4-7]
	\arrow["Fm", from=3-6, to=4-7]
	\arrow["{r_j}", from=1-7, to=2-7]
	\arrow["{r_i}"', from=1-7, to=2-5]
	\arrow["{q_j}", curve={height=-30pt}, from=1-7, to=4-7]
\end{tikzcd}\]
We will find $v$ making that diagram commutative.
The right-hand diagram shows that $q_j$ factorizes through $Fm$. This implies $m_j\leq m$ (by the minimality of $m_j$). Therefore
$$\gr  M\geq \gr  M_j=k.$$
But the strong epimorphism $e\colon M_i\to M$ yields
$$\gr  M\leq \gr  M_i=k,$$
hence $\gr M=k$. Thus $m$ and $m_j$ represent the same subobject of $D_j$: $m=m_j\cdot v$ for some isomorphism $v\colon M\to M_j$. The desired morphism is
$$m_{ij}=v\cdot e.$$
Indeed, $m_j\cdot m_{ij}=m_j\cdot v\cdot e=m\cdot e=d_{ij}\cdot m_i$. Moreover, since $e$ is a strong epimorphism, so is $m_{ij}$, and thus, since $M_i$ and $M_j$ have the same grade, $m_{ij}$ is invertible $(i\leq j \, \text{ in } I_k)$ (Remark \ref{rem:rank}).

The   codirected  diagram $\hat{D}$ of objects $M_i$ and morphisms $m_{ij}\; (i\leq j \, \text{ in } I_k)$ has invertible connecting morphisms, hence its limit (with invertible limit maps) is clearly absolute. Thus $F$ preserves it. The morphisms $r_i\colon Q\to FM_i$ form a cone of $F\hat{D}$: in the right-hand diagram above the upper part commutes because $Fm_j$ is monic, and by post-composing by $Fm_j$ one gets $q_j=Fd_{ij}\cdot q_i$. If $\hat{\pi}_i\colon \hat{L}\to M_i \, (i\in I_k)$ is a limit of $\hat{D}$, we obtain a unique morphism
$$r\colon Q\to F\hat{L}\, \text{ with } r_i=F\hat{\pi}_i\cdot r\; \; (i\in I_k).$$
The natural transformation from $\hat{D}$ to $D_0$ with components $m_i\colon M_i\to D_i\, (i\in I_k)$ yields (since $D$ and $D_0$ have the same limit) a morphism $s\colon \hat{L}\to L$ with $m_i\cdot \hat{\pi}_i=\pi_i\cdot s\; (i\in I_k)$. The desired factorization of $(q_i)$ through $(F\pi_i)$ is given by
$$Fs\cdot r\colon Q\to FL.$$
Indeed, for $i\in I_k$ we have $F\pi_i\cdot (Fs\cdot r)=Fm_i\cdot F\hat{\pi}_i\cdot r= Fm_i\cdot r_i=q_i$.

\vspace{1mm}

(2b) Uniqueness. Given $u,v\colon Q\to FL$ merged by $F\pi_i$ for every $i\in I$, we prove $u=v$.   Form the directed colimit of all  finitely presentable  subobjects $m\colon M \rightarrowtail L$ of $L$ in $\catk$ (see Lemma \ref{lem:fp_subobjs}).
Both $u$ and $v$ factorize through $Fm$ for one of these subobjects, since $F$ preserves that directed colimit and $Q\in \catl_{fp}$. Let $u',\, v'$ be the corresponding factorizations:
\[\begin{tikzcd}
	Q && FL && {FD_i} \\
	&& FM
	\arrow["v"', shift right, from=1-1, to=1-3]
	\arrow["u", shift left=2, from=1-1, to=1-3]
	\arrow["{u'}"{pos=0.7}, shift right, from=1-1, to=2-3]
	\arrow["{v'}"'{pos=0.6}, shift right=4, from=1-1, to=2-3]
	\arrow["{F\pi_i}", dashed, from=1-3, to=1-5]
	\arrow["Fm"', from=2-3, to=1-3]
\end{tikzcd}\]
The proof of $u=v$ will be finished when we verify that there exists $i\in I$ such that ${\pi}_i \cdot m$ is monic. Indeed, then $F{\pi}_i \cdot Fm$ is monic, thus $u' =v'$, which implies $u=v$.

We proceed analogously to Item (2a). For each $i \in I$ we find the least subobject $\bar{m}_i \colon\bar{M}_i \to D_i$ through which the composite $\pi_i \cdot m$ factorizes:
\[\begin{tikzcd}
	M & {\bar{M}_i} \\
	L & {D_i}
	\arrow["{\bar{\pi}_i}", from=1-1, to=1-2]
	\arrow["m"', from=1-1, to=2-1]
	\arrow["{\pi_i}"', from=2-1, to=2-2]
	\arrow["{\bar{m}_i}", from=1-2, to=2-2]
\end{tikzcd}\]
We conclude that there exists an initial subset $I_k \subseteq I$ such that all $\bar{M}_i$ for $i \in I_k$ have the same grade.

Next for each $i \leq j$ in $I_k$ we factorize $d_{ij} \cdot \bar{m}_i$ as a strong epimorphism $\bar{e} : \bar{M}_i \to \bar{M}$ followed by a monomorphism $\bar{m}$:
\[\begin{tikzcd}
	& M \\
	{\bar{M}_i} && {\bar{M}_j} \\
	& {\bar{M}} \\
	{D_i} && {D_j}
	\arrow["{d_{ij}}"', from=4-1, to=4-3]
	\arrow["{\bar{m}}", tail, from=3-2, to=4-3]
	\arrow["{\bar{m}_i}"', from=2-1, to=4-1]
	\arrow["{\bar{m}_j}", from=2-3, to=4-3]
	\arrow["v", dashed, from=3-2, to=2-3]
	\arrow["{\bar{e}}", two heads, from=2-1, to=3-2]
	\arrow["{\bar{\pi}_i}"', from=1-2, to=2-1]
	\arrow["{\bar{\pi}_j}", from=1-2, to=2-3]
\end{tikzcd}\]
	We conclude that ${\pi}_j \cdot m$ ($=\bar{m}_j\cdot \bar{\pi}_j$) factorizes through $\bar{m}$. Arguing as in (2a), we obtain a morphism $v$ such that the above diagram commutes.
	For the morphism ${\bar{m}}_{ij}= v \cdot \bar{e}$ we get the following commutative square
	\[\begin{tikzcd}
		{\bar{M}_i} & {\bar{M}_j} \\
		{D_i} & {D_j}
		\arrow["{d_{ij}}"', from=2-1, to=2-2]
		\arrow["{\bar{m}_i}"', from=1-1, to=2-1]
		\arrow["{\bar{m}_j}", from=1-2, to=2-2]
		\arrow["{\bar{m}_{ij}}", from=1-1, to=1-2]
	\end{tikzcd}\]
	Moreover each ${\bar{m}}_{ij}$ is invertible.
	
	This defines a diagram $\hat{D}$ of all $\bar{M}_i \; (i\in I_k)$. Let $\hat{L}$ be its limit with (invertible) limit maps $\hat{\pi}_i$. This yields the following morphisms:
	$$s\colon \hat{L} \to L;\,  {\pi}_i \cdot s = \bar{m}_i\cdot \hat{\pi}_i$$ and $$r\colon M \to \hat{L};\,  \hat{\pi}_i \cdot r =  \bar{\pi}_i.$$

In the following diagram
\[\begin{tikzcd}
	M \\
	& {\hat{L}} & {\bar{M}} && {} \\
	& L & {D_i} && {}
	\arrow["m"', curve={height=6pt}, from=1-1, to=3-2]
	\arrow["r", from=1-1, to=2-2]
	\arrow["{\bar{\pi}_i}", curve={height=-6pt}, from=1-1, to=2-3]
	\arrow["{\hat{\pi}_i}", from=2-2, to=2-3]
	\arrow["s"', from=2-2, to=3-2]
	\arrow["{\pi_i}"', from=3-2, to=3-3]
	\arrow[""{name=0, anchor=center, inner sep=0}, "{\bar{m}_i}", from=2-3, to=3-3]
	\arrow[""{name=1, anchor=center, inner sep=0}, draw=none, from=2-5, to=3-5]
	\arrow["{(i\in I_k)}"{description}, draw=none, from=0, to=1]
\end{tikzcd}\]
the square and the upper triangle commute. So does the outward shape.
This proves that the left-hand triangle also commutes: use that all $\pi_i$ are collectively monic, because $I_k$ is an initial subset. Since $m$ is monic, we conclude that $r$ is also monic. Now $\hat{\pi}_i$ is invertible and ${\bar{m}}_i$ is monic for each $i \in I_k$, thus the following morphism$${\pi}_i \cdot m = \bar{m}_i \cdot \hat{\pi}_i \cdot r$$ is monic.
\end{proof}

\begin{example}
	Preservation of {\em finite} limits is not sufficient for being a right adjoint even for finitary set functors. Indeed, consider the subfunctor
	$$H\hookrightarrow (-)^{\mathds{N}}$$
	assigning to every set $X$ the set $HX$ of all sequences $a\colon \mathds{N}\to X$ that are eventually constant: there is $n\in \mathds{N}$ with $a(n)=a(m)$ for all $m\geq n$. Then $H$ clearly preserves finite products: a sequence in $X\times Y$ is eventually constant iff both of its projections (to $X^{\mathds{N}}$ and $Y^{\mathds{N}}$) are. $H$ also preserves equalizers. However, $H$ does not preserve the product
	$$A=\prod_{n\in \mathds{N}} A_n\, \text{ where } A_n=\{0,1,\dots, n\}.$$
	Indeed, $HA_n$ contains the sequence $s_n=(0,1,\dots, n,n, n, \dots)$. Thus $(s_n)_{n\in \mathds{N}}\in \Pi_{n\in \mathds{N}}HA_n$. But no element of $HA$ corresponds to $(s_n)$.
	
\end{example}

\begin{remark}\label{rem:metric}
The theorem above can be extended beyond locally finitely presentable categories. This enables us adding to our list of examples categories such as metric spaces or complete metric spaces.

Let $\Met$ be the category of extended metric spaces (i.e., we allow the distance $\infty$) and non-expansive maps. This category is not locally finitely presentable: no non-empty space is finitely presentable \cite[Rem. 2.7]{AR1}.
However, a slight modification on the conditions (1)-(6) of Remark \ref{rem:graduated}, with  finite spaces in the place of finitely presentable objects, allows us to recapture the proof of Theorem \ref{thm:main} for finitary endofunctors of $\Met$ (see Proposition \ref{pro:met} below).
The grades are simple: we use the cardinality of the finite space.

Analogously, for the full subcategory of complete spaces $\CMet$ the choice of finite (thus complete) spaces works. Directed colimits are described in \cite[Prop. 6.3]{ADV}.

\end{remark}

\begin{lemma}In $\Met$ and $\CMet$ regular monomorphisms are precisely the closed isometric embeddings.
	\end{lemma}
	
	\begin{proof} Every regular monomorphism in $\Met$ or $\CMet$ is a closed isometric embedding. Indeed, for two morphisms $f, g\colon B\to C$ the subspace $A=\{b\in B; \, f(b)=g(b)\}$ of $B$ is closed, and the inclusion map $e\colon A\to B$ is an equalizer of $f$ and $g$.
		
		Conversely, let $e\colon A\to B$ be a closed isometric embedding. Without loss of generality, $A$ is a subspace of $B$ and $e$ is the inclusion map. Define a space $C$ by the following pushout
\[\begin{tikzcd}
	& A \\
	B && B \\
	& C
	\arrow["e"', hook', from=1-2, to=2-1]
	\arrow["e", hook, from=1-2, to=2-3]
	\arrow["{{m_0}}"', from=2-1, to=3-2]
	\arrow["{{m_1}}", from=2-3, to=3-2]
\end{tikzcd}\]
	We can describe $C$ as the set $A+(B-A)\times\{0,1\}$ with the following metric $d_C$: for $i=0,1$ the subspace $A+(B-A)\times\{i\}$ carries the metric determined by (the obvious isomorphism to) the space $(B,d_B)$; elements $(x,0)$ and $(y,1)$ with $x,y\in B-A$ have distance 
	$$d_C((x,0), (y,1))=\inf_{a\in A}\{d_B(x,a)+d_B(a,y)\}.$$
	Since $A$ is closed, $d_C((x,0),(y,1))\not=0$. It is easy to verify that $d_C$ is a well-defined metric, and that the obvious embeddings
	$$m_i\colon B\to C \;\; (i=0,1)$$
	form a pushout of $e$ with itself. 
	
	Clearly
	 the embedding $e$ is the equalizer of $m_0$ and $m_1$.
		\end{proof}

\begin{proposition}\label{pro:met}
	A finitary endofunctor of $\Met$ or $\CMet$ is a right adjoint iff it preserves countable limits.
\end{proposition}

\begin{proof}
We present a proof for $\Met$, that for $\CMet$ is analogous.

	We first need to establish some properties which show that, in a sense, finite spaces can substitute finitely presentable objects.
	
	(i)
	In $\Met$ epimorphisms are the morphisms with a dense image. Thus $\Met$ has the (epi, regular mono) factorization system, see \cite[Ex. 3.16]{AR1}.
	Observe that regular monomorphisms into $B$ with finite domains precisely represent the finite subspaces of $B$.
	
	(ii) Every space is a canonical directed colimit of the diagram of its finite subspaces. The colimit maps and connecting morphisms are regular monomorphisms. For a description of directed colimits see \cite[Prop. 2.9]{ADV}.
	
		(iii) Let $D$ be a directed diagram of finite spaces with connecting maps regular monic.
	Then every morphism $f\colon M\to \text{colim}\, D$, where $M$ is a finite space, factorizes through some colimit map. Indeed, using (i) and (ii) we can assume that for the collection $D_i\, (i\in I)$ of objects of $D$, given $i\leq j$ in $I$ the connecting map $D_i\to D_j$ is the inclusion map of a subspace of $D_j$. Then $\text{colim}\, D$ is simply the union $\displaystyle{\bigcup_{i\in I}D_i}$ with the induced metric.
	For $f\colon M\to \displaystyle{\bigcup_{i\in I}D_i}$ there exists $j\in J$ with $f[M]
\subseteq D_j$. Since $f$ is nonexpanding, and $D_j$ is a subspace of $\displaystyle{\bigcup_{i\in I}D_i}$, it follows that the codomain restriction of $f$ to $f'\colon M\to D_j$ is nonexpanding. This is the desired factorization through the colimit map  $D_j\hookrightarrow \text{colim}\, D$.

We are ready to follow the steps of the proof of Theorem \ref{thm:main}.

 (1) We only need to prove that the given finitary endofunctor
 $F$ preserves countably codirected limits. Then it preserves limits. Now $\Met$ has a cogenerator $\mathds{R}$ (with the Euclidean metric). Indeed, for every space $X$ and every element $x\in X$ the distance function
 $$d(x,-)\colon X\to \mathds{R}$$
 is nonexpanding. Since $d(x,-)\not=d(y,-)$ whenever $x\not= y$,  $\mathds{R}$ cogenerates $\Met$. By the Special Adjoint Functor Theorem $F$ is a right adjoint.

 (2) Let $l_i\colon L\to D_i (i\in I)$ be a countably codirected limit of a diagram $D$. Using (ii) above, it is sufficient to prove for every finite space $Q$ that each cone $q_i\colon Q\to FD_i\, (i\in I)$ uniquely factorizes through $(Fl_i)_{i\in I}$.

 (2a) Existence. For every space $K$ and every morphism $q\colon Q\to FK$ there exists a least subspace $m\colon M\rightarrowtail K$ with $M$  finite such that $q$ factorizes through $Fm$. This follows from (ii) above and $F$ preserving directed colimits and pullbacks, precisely as in the proof of Theorem \ref{thm:main}.
 We thus obtain for each $i\in I$ the least finite subspace $m_i\colon M_i\hookrightarrow D_i$ with $q_i=Fm_i\cdot q_i^{\prime}$.  Put $I_n=\{i\in I;\, \text{card}\, M_i=n\}$. Some $I_k$ is initial in $I$. The argument that for $i\leq j$ in $I_k$, we have $m_{ij}\colon M_i\to M_j$ with $d_{ij}\cdot m_i=m_j\cdot m_{ij}$ is as above, just using the (epi, regular mono) factorizations. We obtain a diagram $\hat{D}$ of all $M_i$, $i\in I_k$, and all $m_{ij}$. The latter are bijections (because they are monic and $\text{card}\, M_i=k=\text{card}\, M_j$), and being regular monomorphisms, they are invertible. The rest is completely analogous to the proof of \ref{thm:main}.

 (2b) Uniqueness. With the modifications of the proof of \ref{thm:main} we have seen in item (2a), the proof of (2b) in \ref{thm:main} works completely analogously.
	\end{proof}

\section{Absolute intersections}

In categories such as $K$-$\Vec$ and $\Set_p$ finite intersections are absolute limits (preserved by all functors). We prove this, using ideas of Trnkov\'a \cite{T1}  who proved that {\em nonempty} intersections in $\Set$ are absolute (see Remark \ref{rem:finite products}).

\begin{definition}
A  category  $\catk$ has {\em absolute intersections} provided that all monomorphisms split, and for every intersection of monomorphisms $m$ and $m'$
\begin{equation}\label{eq:(1)}
\begin{tikzcd}
	& C \\
	B & {} & {B'} \\
	& A
	\arrow["{m'}", tail, from=2-3, to=3-2]
	\arrow[""{name=0, anchor=center, inner sep=0}, "{i'}", tail, from=1-2, to=2-3]
	\arrow["{e'}"', curve={height=18pt}, dashed, two heads, from=2-3, to=1-2]
	\arrow["m"', tail, from=2-1, to=3-2]
	\arrow["e", curve={height=-18pt}, dashed, two heads, from=3-2, to=2-1]
	\arrow[""{name=1, anchor=center, inner sep=0}, "i"', tail, from=1-2, to=2-1]
	\arrow[shift left=3, shorten <=4pt, shorten >=4pt, no head, from=1, to=2-2]
	\arrow[shift right=3, shorten <=4pt, shorten >=4pt, no head, from=0, to=2-2]
\end{tikzcd}
\end{equation}
there exist splittings $e$ of $m$ and $e'$ of $i'$ with
\begin{equation}\label{eq:(2)}
e\cdot m'=i\cdot e'\colon B'\to B.
\end{equation}
\end{definition}

\begin{proposition}\label{pro:absolute} The pullback in the above definition is absolute.
\end{proposition}

\begin{proof} Given a functor $F\colon \catk \to \catl$ and a commutative square in $\catl$ as follows
\begin{equation}\label{eq:(3)}\begin{tikzcd}
	& U \\
	FB && {FB'} \\
	& FA
	\arrow["u"', from=1-2, to=2-1]
	\arrow["{u'}", from=1-2, to=2-3]
	\arrow["Fm"', from=2-1, to=3-2]
	\arrow["{Fm'}", from=2-3, to=3-2]
\end{tikzcd}
\end{equation}
we prove that the desired factorization of $(u,u')$ through $(Fi, Fi')$ is
$$v=Fe'\cdot u'\colon U\to FC.$$
The uniqueness is clear since $Fi$ is monic. Our task is to verify that the diagram below commutes:
\[\begin{tikzcd}
	& U \\
	& {FB'} \\
	FB & FC & {FB'}
	\arrow["u"', curve={height=6pt}, from=1-2, to=3-1]
	\arrow["{u'}", curve={height=-6pt}, from=1-2, to=3-3]
	\arrow["Fi", from=3-2, to=3-1]
	\arrow["{Fi'}"', from=3-2, to=3-3]
	\arrow["{u'}"', from=1-2, to=2-2]
	\arrow["{Fe'}"', from=2-2, to=3-2]
\end{tikzcd}\]
The left-hand triangle does:
$$\begin{array}{rlr}Fi\cdot (Fe'\cdot u')&=Fe\cdot Fm'\cdot u'& \text{by \eqref{eq:(2)}}\\
&=Fe\cdot Fm\cdot u &\text{by \eqref{eq:(3)}}\\
&=u& \text{as $e\cdot m=\id$.}
\end{array}$$
The right-hand triangle commutes because $Fm'$ is monic, and we have
$$\begin{array}{rlr}Fm'\cdot u'&=Fm\cdot u& \text{by \eqref{eq:(3)}}\\
&=Fm\cdot Fe\cdot Fm\cdot u& \text{as $e\cdot m=\id$}\\
&=Fm\cdot Fe\cdot Fm'\cdot u'&\text{by \eqref{eq:(3)}}\\
&=Fm\cdot Fi\cdot Fe'\cdot u'&\text{by \eqref{eq:(2)}}\\
&=Fm'\cdot (Fi'\cdot Fe'\cdot u') &\text{by \eqref{eq:(1)}}.
\end{array}$$
\end{proof}

\begin{examples}\label{exa:absolute} (1) The  category  $K$-$\Vec$ has absolute intersections. Without loss of generality we assume that in the pullback \eqref{eq:(1)} the objects fulfil
$$B\subseteq A, \, B'\subseteq A \; \text{and}\; C=B\cap B',$$
and the morphisms are the inclusion maps. We decompose the spaces $B$ and $B'$ as follows:
$$B=B_0\oplus C\; \text{ and } B'=B'_0\oplus C.$$
Then $A$ has the following decomposition:
$$A=A_0 \oplus B_0 \oplus B'_0\oplus C.$$
The desired splittings are as follows:
\[\begin{tikzcd}
	& C \\
	{B_0\oplus C} && {B'_0\oplus C} \\
	& {A_0\oplus B_0\oplus B'_0\oplus C}
	\arrow[tail, from=1-2, to=2-1]
	\arrow[tail, from=1-2, to=2-3]
	\arrow["{[0,\id]}"', shift right, curve={height=12pt}, from=2-3, to=1-2]
	\arrow[tail, from=2-1, to=3-2]
	\arrow["{[0,\id,0,\id]}"{pos=0.6}, shift left=2, curve={height=-12pt}, from=3-2, to=2-1]
	\arrow[tail, from=2-3, to=3-2]
\end{tikzcd}\]

(2) The  category  $\Set_p$ has absolute intersections. Without loss of generality we assume that, again, the morphisms in the pullback \eqref{eq:(1)} are inclusion maps. In particular, all four objects have the same specified element $c\in C$. Define $e\colon (A,c)\to (B,c)$ and $e'\colon (B',c)\to (C,c)$ by
$$e(x)=\left\{\begin{array}{ll}y&\text{if $y\in B$}\\
c&\text{else}
\end{array}\right.  \hspace{15mm} e'(z)=\left\{\begin{array}{ll}z&\text{if $z\in C$}\\
c&\text{else}.
\end{array}\right.$$
These are the desired splittings.
\end{examples}

\begin{remark}\label{rem:finite products} (1) We conclude that an endofunctor of $K$-$\Vec$ or $\Set_p$ preserves (finite) products iff it preserves (finite) limits. This also follows from results presented by Trnkov\'a in \cite{T3} (Prop. 4 and Example B). In that paper Trnkov\'a studies categories $\catk$  such that  every functor with domain $\catk$ preserving products preserves limits. Besides vector spaces and pointed sets, Trnkov\'a shows that examples of such categories $\catk$ include sets with monomorphisms and topological $T_1$-spaces with closed maps.

(2) Every nonempty finite intersection in $\Set$ is absolute. Indeed, this is analogous to the case $\Set_p$: given subsets $m\colon B\hookrightarrow A$ and $m'\colon B'\hookrightarrow A$ with $c\in B\cap B'$, we define $e\colon A\to B$ and $e'\colon B' \to B\cap B'$ as in Example \ref{exa:absolute}(2).
Preservation of nonempty intersections was proved by Trnkov\'a (cf. \cite[Proposition 2.1]{T1}).
\end{remark}

\section{Set functors preserving countable products}

We have seen that for endofunctors of $\Set_p$ there is no difference between preservation of countable products and countable limits. Is the same true for $\Set$? Not quite:

\begin{example} The functor $C_{01}$ given by $C_{01}\emptyset=\emptyset$ and $C_{01}X=1$ for all $X\not=\emptyset$ clearly preserves products. But it does not preserve the intersection of the coproduct injections of $1+1$:
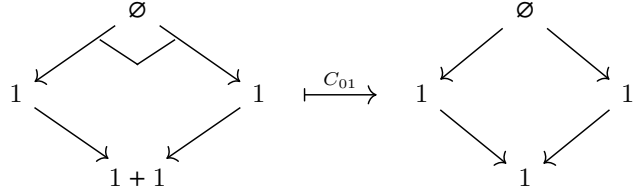
\[\begin{tikzcd}
	& \emptyset &&&& \emptyset \\
	1 & {} & 1 && 1 && 1 \\
	& {1+1} &&&& 1
	\arrow[""{name=0, anchor=center, inner sep=0}, from=1-2, to=2-1]
	\arrow[""{name=1, anchor=center, inner sep=0}, from=1-2, to=2-3]
	\arrow[from=2-1, to=3-2]
	\arrow[from=2-3, to=3-2]
	\arrow["{C_{01}}", shorten <=10pt, shorten >=10pt, maps to, from=2-3, to=2-5]
	\arrow[from=1-6, to=2-5]
	\arrow[from=1-6, to=2-7]
	\arrow[from=2-5, to=3-6]
	\arrow[from=2-7, to=3-6]
	\arrow[shift left=4, shorten <=4pt, shorten >=4pt, no head, from=0, to=2-2]
	\arrow[shift right=4, shorten <=4pt, shorten >=4pt, no head, from=1, to=2-2]
\end{tikzcd}\]
This is the unique such set functor (up to natural isomorphism), as we now prove.
\end{example}

\begin{definition}(Trnkov\'a \cite{T2}) Let $H$ be a set functor. An element $x\in HX$ is {\em distinguished} if for all parallel pairs $f,g\colon X\to Y$ we have $Hf(x)=Hg(x)$.
\end{definition}

\begin{example}\label{exa:distinguished} (1) Every element $x\in H\emptyset$ is distinguished.

(2) If $x\in HX$ is distinguished, so is $Hf(x)\in HY$ for each $f\colon X\to Y$.
\end{example}

The following result can be derived from \cite[Prop. I.4]{T1} and \cite[Prop. II.6]{T2}.  We present a short proof for the convenience of the reader.

\begin{proposition}\label{pro:distinguished}
Every set functor without distinguished elements preserves finite intersections.
\end{proposition}

\begin{proof} Let $H\colon \Set \to \Set$ have no distinguished element. By the above example, $H\emptyset=\emptyset$. We already know from Remark \ref{rem:finite products} that $H$ preserves nonempty intersections. Thus we only need to consider disjoint subsets $A_1,\, A_2$ of $B$:
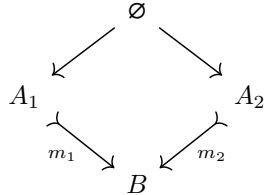
\[\begin{tikzcd}
	& \emptyset \\
	{A_1} && {A_2} \\
	& B
	\arrow[from=1-2, to=2-1]
	\arrow[from=1-2, to=2-3]
	\arrow["{m_1}"', tail, from=2-1, to=3-2]
	\arrow["{m_2}", tail, from=2-3, to=3-2]
\end{tikzcd}\]
Suppose $H$ does not preserve this pullback, then we prove that it has a distinguished element. Since $H\emptyset=\emptyset$ and $H$ does not preserve the above pullback, there exist $t_i\in HA_i$ with $Hm_1(t_1)=Hm_2(t_2)=t$. The element $t\in HB$ is distinguished. Indeed, for every pair $f,g\colon B\to Y$ we can choose a map $h\colon B\to Y$ coinciding on $A_1$ with $f$ and on $A_2$ with $g$:
$$h\cdot m_1=f\cdot m_1\; \text{ and }\; h\cdot m_2=g\cdot m_2.$$
Then
$$Ff(t)=F(f\cdot m_1)(t_1)=Fh(Fm_1(t_1))=Fh(t)$$
as well as
$$Fg(t)=F(g\cdot m_2)(t_2)=Fh(Fm_2(t_2))=Fh(t).$$
\end{proof}

\begin{theorem}
Every set functor $H\not\simeq C_{01}$ preserving finite products preserves finite limits.
\end{theorem}

\begin{proof}
Let $H$ preserve finite products.
We know that $H1\simeq 1$, and we put
$$H1=\{a_1\}.$$
Since $H$ preserves the product $\emptyset =\emptyset\times \emptyset$, we have
$$H\emptyset=\emptyset\; \text{ or }\; H\emptyset\simeq 1.$$

(a) Let $H\emptyset$ contain an element $a_0$. Then, by Example \ref{exa:distinguished}, the element $a_1=Ht(a_0)$ (for the unique $t\colon \emptyset\to 1$) is distinguished. For every set $X\not=\emptyset$ we put
$$a_X=Hf(a_1)\; \text{for}\; f\colon 1\to X$$
and prove
$$HX=\{a_X\}\; \text{ for all $X$.}$$
Thus $H$ is naturally isomorphic to the constant functor of value $1$, and preserves limits.

Our equation $HX=\{a_X\}$ holds for $\emptyset$ and $1$, so we can assume that card$\, X\geq 2$. We first observe that $H$ maps every constant function $f\colon X\to Y$ of value $y$, $f=\text{const}\,y$, to the constant function of value $a_Y$:
$$H(\text{const}\, y)=\text{const}\, a_Y.$$
Indeed, we have $f'\colon 1\to Y$ making the left-hand triangle below commutative
\[\begin{tikzcd}
	& 1 &&&& {\{a_1\}} \\
	X && Y && HX && HY
	\arrow["{!}", from=2-1, to=1-2]
	\arrow["{f'}", from=1-2, to=2-3]
	\arrow["f"', from=2-1, to=2-3]
	\arrow["{H!}", from=2-5, to=1-6]
	\arrow["Hf"', from=2-5, to=2-7]
	\arrow["{Hf'}", from=1-6, to=2-7]
\end{tikzcd}\]
Thus the right-hand triangle verifies the statement: since $a_1$ is distinguished, $Hf'(a_1)=a_Y$.
Choose $x_1\not=x_2$ in $X$ and put
$$f_i=\langle \id, \text{const}\, x_i\rangle\colon X\to X\times X\; \; (i=1,2).$$
The projections $\pi_l,\, \pi_r$ make the following diagrams commutative for $i=1,2$:
\[\begin{tikzcd}
	& X &&&& HX \\
	X & {X\times X} & X && HX & {H(X\times X)} & HX
	\arrow["\id"', from=1-2, to=2-1]
	\arrow["{f_i}", from=1-2, to=2-2]
	\arrow["{\text{const}\, x_i}", from=1-2, to=2-3]
	\arrow["{\pi_l}", from=2-2, to=2-1]
	\arrow["{\pi_r}"', from=2-2, to=2-3]
	\arrow["\id"', from=1-6, to=2-5]
	\arrow["{Hf_i}", from=1-6, to=2-6]
	\arrow["{\text{const}\, a_X}", from=1-6, to=2-7]
	\arrow["{H\pi_l}", from=2-6, to=2-5]
	\arrow["{H\pi_r}"', from=2-6, to=2-7]
\end{tikzcd}\]
Since $H$ preserves $X\times X$, the pair $H\pi_l$, $H\pi_r$ is collectively monic. This proves
$$Hf_1=Hf_2\colon HX\to H(X\times X).$$
Next consider the following map
$$g\colon X\times X\to X, \; g(u,v)=\left\{\begin{array}{ll}x_1&\text{if $v=x_1$}\\
u&\text{else.}\end{array}\right.$$
Then the diagram below commutes:
\[\begin{tikzcd}
	X & {X\times X} & X \\
	& X
	\arrow["{\text{const}\, x_1}"', from=1-1, to=2-2]
	\arrow["\id", from=1-3, to=2-2]
	\arrow["g", from=1-2, to=2-2]
	\arrow["{f_1}", from=1-1, to=1-2]
	\arrow["{f_2}"', from=1-3, to=1-2]
\end{tikzcd}\]
Apply $H$ to it and get (using $H(\text{const}\, x_1)=\text{const}\, a_X$) that
$$\id_{HX}=\text{const}\, a_X.$$
This proves $HX=\{a_X\}.$

(b) Let $H\emptyset=\emptyset$. If $a_1\in H1$ is distinguished, then, as in (a), we derive $HX=\{a_X\}$ for all $X\not=\emptyset$. Thus $H$ is naturally isomorphic to $C_{01}$.

If $a_1$ is not distinguished, then $H$ has no distinguished element ($a\in HX$ distinguished implies $Hf(a)$ distinguished for $f\colon X\to 1$). Apply Proposition \ref{pro:distinguished} to conclude that $H$ preserves finite limits.
\end{proof}

\begin{corollary} A finitary set functor $H\not\simeq C_{01}$ is a right adjoint if and only if it preserves countable products.
\end{corollary}

\bibliographystyle{abbrv}

\begin{thebibliography}{99}

 \bibitem{AR} J. Ad\'amek and J. Rosick\'y, {\em Locally presentable and accessible categories}, Cambridge Univ. Press, Cambridge 1994.



\bibitem{AR1} J. Ad\'amek and J. Rosick\'y, Approximate injectivity and smallness in metric-enriched categories, {\em J. Pure Appl. Algebra} 226 (2022), 1-30.

\bibitem{ADV} J. Ad\'amek, M. Dost\'al and J. Velebil, Quantitative algebras and a classification of metric monads, arXiv:2210.01565.


 \bibitem{GT} G. Tendas, On continuity of accessible functors, {\em Appl. Categ. Structures} 30 (2022), no. 5, 937–946.




\bibitem{T1} V. Trnkov\'a, Some properties of set functors, {\em Comment. Math. Univ.  Carolinae} 10 (1969), 323-352.



         \bibitem{T3} V. Trnkov\'a, When the product-preserving functors preserve limits, {\em Comment. Math. Univ.  Carolinae} 11 (1970), 365-378.

          \bibitem{T2} V. Trnkov\'a, On descriptive classification of set-functors I, {\em Comment. Math. Univ.  Carolinae} 12 (1971), 143-174.


\end{thebibliography}

\end{document}